\documentclass[11pt,a4paper]{amsart}
\usepackage{texmac, hyperref}
\usepackage[all]{xy}
\usepackage{amsmath}
\xyoption{web}

\operators{Aut,Irr,Stab,lcm}
\calsymbols{c}{O,C,B,K}
\bbsymbols{}{}
\fraksymbols{f}{F,p,P,q}

\begin{document}
\title{Factor equivalence of Galois modules and regulator constants}
\author{Alex Bartel}
\thanks{This research is partly supported by a research fellowship from the
Royal Commission for the exhibition of 1851.}
\date{\today}
\address{Department of Mathematics, Warwick University,
Coventry CV4 7AL, UK}
\email{a.bartel@warwick.ac.uk}

\maketitle
\begin{abstract}
We compare two approaches to the study of Galois module structures: on the one hand
factor equivalence, a technique that has been used by Fr\"ohlich and others to
investigate the Galois module structure of rings of integers of number fields
and of their unit groups, and on the other hand regulator constants, a set of
invariants attached to integral group representations by Dokchitser and Dokchitser,
and used by the author, among others, to study Galois module structures. We show
that the two approaches are in fact closely related, and interpret results
arising from these two approaches in terms of each other. We then use this
comparison to derive a factorisability result on higher
$K$-groups of rings of integers, which is a direct analogue of a theorem of
de Smit on $S$-units.
\end{abstract}

\section{Introduction}
Let $G$ be a finite group. Factor equivalence of finitely generated $\Z$-free
$\Z[G]$-modules is an equivalence relation that is a weakening of local
isomorphism. It has been used e.g. in \cite{Frohl,Solomon,deSmit} among many
other works to derive restrictions on the Galois module structure of rings of
integers of number fields and of their units in terms of
other arithmetic invariants.

More recently, a set of rational numbers has been attached to any finitely
generated $\Z[G]$-module, called regulator constants \cite{squarity}, with the
property that if two modules are locally isomorphic, then they have
the same regulator constants. These invariants have been used in \cite{classnorels}
and in \cite{SB} to investigate the Galois module structure of integral units of
number fields, of higher $K$-groups of rings of integers, and of Mordell-Weil
groups of elliptic curves over number fields.

It is quite natural to ask whether there is a connection between the two approaches
to Galois modules and whether the results of one can be interpreted in terms
of the other. It turns out that there is indeed a strong connection, which we
shall investigate here. We will begin in the next section by recalling the
definitions of factorisability, of factor equivalence, and of regulator
constants. We will then establish some purely algebraic results that link
factor equivalence and regulator constants. In \S\ref{sec:results} we will
revisit the relevant results of \cite{Frohl,deSmit,classnorels,SB} on Galois
module structures and will use the link established in \S\ref{sec:conn} to
compare them to each other. Finally, in \S\ref{sec:K} we will use the results
of \S\ref{sec:conn} to prove a
factorisability result on $K$-groups of rings of integers that is a direct
analogue of \cite[Theorem 5.2]{deSmit}.

Throughout the paper, whenever there will be mention of a group $G$, we will
always assume it to be finite. All $\Z[G]$-modules
will be assumed to be finitely generated and all representations will be
finite-dimensional.

\begin{acknowledgements}
This work is partially funded by a research fellowship from the Royal
Commission for the Exhibition of 1851. Part of this work was done, while
I was a member of the Mathematics Department at Postech, Korea. It is
a great pleasure to thank both institutions for financial support, and
to thank Postech for a friendly and supportive working environment. I am also
grateful to an anonymous referee for carefully reading the manuscript and for
helpful suggestions.
\end{acknowledgements}

\section{Factorisability and regulator constants}\label{sec:conn}
\subsection{Factorisability and factor equivalence}
We will begin by recalling the definition of factorisability and of factor
equivalence, and by discussing slight reformulations.
This concept first appears in \cite{Nelson} and plays a prominent r\^ole e.g. in
the works of Fr\"ohlich.

\begin{definition}\label{def:factorisability}
Let $G$ be a group (always assumed to be finite), and let $X$ be an abelian
group, written multiplicatively.
A function $f: H\mapsto x\in X$ on the set of subgroups $H$ of $G$ with values in
$X$ is \emph{factorisable} if there exists an injection of abelian groups
$\iota: X\hookrightarrow Y$ and a function $g: \chi\mapsto y\in Y$ on the
irreducible characters of $G$ with values in $Y$, with the property that
\[
\iota(f(H)) = \prod_{\chi\in\Irr(G)} g(\chi)^{\langle\chi,\C[G/H]\rangle}
\]
for all $H\leq G$, where $\Irr(G)$ denotes the set of irreducible characters of
$G$, and $\langle\cdot,\cdot\rangle$ denotes the usual inner product of characters.
\end{definition}

The definition one often sees in connection with Galois module structures is
a special case of this: $X$ is usually taken to be the multiplicative group of
fractional ideals of the ring of integers $\cO_k$ of some number field $k$,
and $Y$ is required to be the ideal group of $\cO_K$ for some finite Galois
extension $K/k$ with Galois group $G$, with $\iota$ being the natural map
$I\mapsto I\cO_K$. 

Let us introduce convenient representation theoretic language to concisely 
rephrase the above definition.
\begin{definition}
The \emph{Burnside ring} $B(G)$ of a group $G$ is the free abelian group
on isomorphism classes $[S]$ of finite $G$-sets, modulo the subgroup generated
by elements of the form
\[
[S]+[T] - [S\sqcup T],
\]
and with multiplication defined by
\[
[S]\cdot [T] = [S\times T].
\]
\end{definition}

\begin{definition}
The \emph{representation ring} $R_{\cK}(G)$ of a group $G$ over the field
$\cK$ is the free abelian group on isomorphism classes $[\rho]$ of finite dimensional
$\cK$-representations of $G$, modulo the subgroup generated by elements of the form
\[
[\rho]+[\tau] - [\rho\oplus \tau],
\]
and with multiplication defined by
\[
[\rho]\cdot [\tau] = [\rho\otimes \tau].
\]
\end{definition}

In the case that $\cK=\Q$, which will be the main case of interest, we will
omit the subscript and simply refer to the representation ring $R(G)$ of $G$.

There is a natural map $B(G)\rightarrow R(G)$ that sends a $G$-set $X$ to the
permutation representation $\Q[X]$. Denote its kernel by $K(G)$. By
Artin's induction theorem, this map always has a finite cokernel $C(G)$
of exponent dividing $|G|$. Moreover, $C(G)$ is known to be trivial in many
special cases, e.g. if $G$ is nilpotent, or a symmetric group. The cokernel $C(G)$
is important when strengthenings of the notion of factorisability are considered,
such as $F$-factorisability, but will not be important for us.

It follows immediately from Definition \ref{def:factorisability} and from
standard representation theory that for $f$ to be factorisable, it has to
be constant on conjugacy classes of subgroups. There is a bijection between
conjugacy classes of subgroups of $G$ and isomorphism classes of transitive
$G$-sets, which assigns to $H\leq G$ the set of cosets $G/H$ with left $G$-action
by multiplication, and to a $G$-set $S$ the conjugacy class of any point stabiliser
$\Stab_G(s)$, $s\in S$. An arbitrary $G$-set is a disjoint union
of transitive $G$-sets, and so an element of $B(G)$ can be identified with
a formal $\Z$-linear combination of conjugacy classes of subgroups of $G$. So if
$f$ is a factorisable function, then it can be thought of as a function on
conjugacy classes of subgroups of $G$, equivalently on transitive $G$-sets, and
then extended linearly to yield a group homomorphism $B(G)\rightarrow X$.

\begin{proposition}\label{prop:factorisability}
Let $f:B(G)\rightarrow X$ be a group homomorphism, where $X$ is an abelian
group. The following are equivalent:
\begin{enumerate}
\item $f$ is factorisable in the sense of Definition \ref{def:factorisability}.
\item There exists an injection $\iota:X\hookrightarrow Y$
  of abelian groups such that the composition $\iota\circ f$ factors through
  the natural map $B(G)\rightarrow R_{\C}(G)$.
\item There exists
  an injection $\iota':X\hookrightarrow Y'$ such that $\iota'\circ f$ factors
  through the natural map $B(G)\rightarrow R(G)$, i.e. there is a homomorphism
  $g':R(G)\rightarrow Y'$ that makes the following diagram (whose first row is
  exact) commute:
  \[
  \xymatrix{
  0\ar[r] & K(G) \ar[r] & B(G) \ar[r] \ar[d]_f & R(G)\ar[d]^{g'} \ar[r] & C(G)\ar[r] & 0\\
  & & X \ar@{^{(}->}[r]_{\iota'} & Y'.& &
  }
  \]
\item The homomorphism $f$ is trivial on $K(G)=\ker(B(G)\rightarrow R(G))$.
\end{enumerate}
\end{proposition}
\begin{proof}
The condition (2) is just a reformulation of (1).

Suppose that condition (2) is satisfied, and let us deduce (3).
Let $g$ be the map $R_{\C}(G)\rightarrow Y$ whose existence is postulated by (2).
Define $Y'$ to be the subgroup of $Y$
generated by $\iota(X)$ and by $g(R(G))$, define $g'$ to be the restriction
of $g$ to $R(G)$, followed by the inclusion $g(R(G))\hookrightarrow Y'$,
and $\iota'$ to be $\iota$, followed by the inclusion $\iota(X)\hookrightarrow Y'$.
Then $Y'$, $\iota'$, $g'$ satisfy (3).

A brief diagram chase shows that (3) implies (4): since $\iota'$, is
an injection, $\ker(\iota'\circ f) = \ker f$. So for the diagram in (3) to
commute, we must have $\ker f \geq \ker(B(G)\rightarrow R(G))=K(G)$. 
Incidentally, exactly the same proof shows also that (2) implies (4).

Finally, the implication (4) $\Rightarrow $(2),(3) follows from two standard facts
about abelian groups:
\begin{itemize}
\item any abelian group can be embedded into a divisible
abelian group,
\item and any homomorphism from a subgroup $A$ of an abelian group
$B$ to a divisible group $D$ extends to a homomorphism from $B$ to $D$.
\end{itemize}
Since $f$ is trivial on $K(G)$, it induces a homomorphism from $B(G)/K(G)$,
which is canonically identified with a subgroup of $R(G)\leq R_\C(G)$. Now,
embed $X$ into a divisible group $Y$, and extend $f:B(G)/K(G)\rightarrow Y$
to a homomorphism $R(G) \hookrightarrow R_\C(G)\rightarrow Y$.
\end{proof}

%

\begin{remark}
\begin{enumerate}
\item
It follows from the last part of the proof that if $X$ is divisible, then $Y'$
can be taken to be equal to $X$ in Proposition \ref{prop:factorisability}.
Also, if $C(G)$ is trivial, then $B(G)/K(G)\cong R(G)$, and again $Y'$ can
be taken to be equal to $X$.
\item If $X$ is the group of fractional ideals
of a number field $k$, and if $f$ vanishes on $K(G)$, then $Y'$ can always be
taken to be the group of fractional ideals of a suitable Galois extension $K/k$,
so this is not an additional restriction. Indeed, a sufficient condition on
$Y'$ is that elements of $B(G)/K(G)$ that are $n$-divisible in $R(G)$ are mapped
under $f$ to elements of $X$ that become $n$-divisible in $Y'$. So if
$X$ is the group of fractional ideals of a number field $k$, this condition
translates into relative ramification indices of some integral
ideals of $K$ being divisible by some integers, and some elements of $k$ having
certain $n$-th roots in $K$.
\end{enumerate}
\end{remark}

\begin{remark}
  In \cite{tamroot}, the word ``representation-theoretic'' has been used in place
  of ``factorisable''.
\end{remark}

\begin{definition}
  Let $G$ be a group, and let $M$, $N$ be two
  $\Z$-free $\Z[G]$-modules such that there is an isomorphism
  of $\Q[G]$-modules $M\otimes \Q\cong N\otimes \Q$. Fix an
  embedding $i: M\rightarrow N$ of $G$-modules with finite cokernel. Then $M$
  and $N$ are said to be \emph{factor equivalent},
  written $M \land N$, if the function $H\mapsto [N^H : i(M^H)]$ is factorisable.
\end{definition}
The notion of factor equivalence is independent of the choice of the embedding
$i$, and defines an equivalence relation on the set of $\Z$-free
$\Z[G]$-modules. If $M\otimes \Z_p\cong N\otimes \Z_p$ for some prime $p$,
then $i$ can be chosen to have a cokernel of order coprime to $p$.
Indeed, $M\otimes \Z_p\cong N\otimes \Z_p$ if and only if
$M\otimes \Z_{(p)}\cong N\otimes \Z_{(p)}$ (\cite{Maranda}, see also
\cite{Reiner}),
and an isomorphism $M\otimes \Z_{(p)}\rightarrow N\otimes \Z_{(p)}$ gives rise
to an embedding $i$ with cokernel of order coprime to $p$, by composing it with
multiplication by an integer to clear denominators.
It follows that two modules that are locally isomorphic at all primes $p$ are
factor equivalent.

The above definition is the one usually appearing in the literature,
but it will be convenient for us to follow \cite{deSmit} in defining factor
equivalence for $\Z[G]$-modules that are not necessarily $\Z$-free:
\begin{definition}
  Let $G$ be a group, and let $M$, $N$ be two
  $\Z[G]$-modules such that there is an isomorphism of $\Q[G]$-modules
  $M\otimes\Q \cong N\otimes \Q$. Fix a map
  $i:M\rightarrow N$ of $G$-modules with finite kernel and cokernel. Then
  $M$ and $N$ are said to be \emph{factor equivalent} if the function
  $H\mapsto [N^H:i(M^H)]\cdot|\ker(i)^H|^{-1}$ is factorisable.
\end{definition}
Again, this notion is independent of the choice of the map $i$, and defines an
equivalence relation on the set of $\Z[G]$-modules that
weakens the relation of lying in the same genus (where $M$ and $N$ are said to
lie in the same genus if $M\otimes \Z_p\cong N\otimes \Z_p$ for all primes $p$).

\subsection{Regulator constants}
We continue to denote by $G$ an arbitrary (finite) group. We also continue to
use the identification between conjugacy classes of subgroups of $G$ and
isomorphism classes of transitive $G$-sets. Under this identification, a general
element of $B(G)$ will be written as $\Theta=\sum_{H\leq G} n_H H$ with the sum
running over mutually non-conjugate subgroups, and with $n_H\in \Z$. An element
of $K(G)$ is such a linear combination with the property that the virtual
permutation representation $\bigoplus_H \Q[G/H]^{\oplus n_H}$ is 0. Alternatively, 
more down to earth, if we write $\Theta$ as $\Theta = \sum_i n_iH_i -
\sum_j n_j'H_j'$ with all $n_i$, $n_j'$ non-negative, then $\Theta$ is in $K(G)$
if and only if the permutation representations $\bigoplus_i \Q[G/H_i]^{\oplus n_i}$
and $\bigoplus_j \Q[G/H_j']^{\oplus n_j'}$ are isomorphic.

\begin{definition}
An element $\Theta = \sum_H n_H H$ of $K(G)$ is called a \emph{Brauer relation}.
\end{definition}

The following invariants of $\Z[G]$-modules were introduced
in \cite{squarity} and used e.g. in \cite{classnorels,SB} to investigate
Galois module structures, as we shall review in the next section:

\begin{definition}
Let $G$ be a group and $M$ a $\Z[G]$--module.
Let $\langle\cdot,\cdot\rangle:M\times M\rightarrow \C$ be a bilinear $G$--invariant
pairing that is non--degenerate on $M/\tors$. Let $\Theta=\sum_{H\leq G} n_H H\in K(G)$
be a Brauer relation.
The regulator constant of $M$ with respect to $\Theta$ is defined by
\[
\cC_{\Theta}(M) = \prod_{H\leq G}\det\left(\frac{1}{|H|}\langle\cdot,\cdot\rangle\big|M^H/\tors\right)\in \C^{\times}.
\]
Here and elsewhere, the abbreviation $\tors$ refers to the $\Z$-torsion subgroup.
\end{definition}
This is independent of the choice of pairing \cite[Theorem 2.17]{tamroot}. As
a consequence, $\cC_{\Theta}(M)$ is always a rational number, since the pairing
can always be chosen to be $\Q$-valued. It is also immediate that
$\cC_{\Theta_1+\Theta_2}(M) = \cC_{\Theta_1}(M)\cC_{\Theta_2}(M)$, so given
a $\Z[G]$-module, it suffices to compute the regulator constants with respect
to a basis of $K(G)$. In other words, this
construction assigns to each $\Z[G]$-module essentially a finite set of rational
numbers, one for each element of a fixed basis of $K(G)$.

One can show that if $M$, $N$ are two $\Z[G]$-modules such that
$M\otimes\Z_p \cong N\otimes \Z_p$, then for all $\Theta\in K(G)$ the $p$-parts of
$\cC_\Theta(M)$ and $\cC_\Theta(N)$ are the same. So, like factor equivalence,
regulator constants provide invariants of a $\Z[G]$-module that, taken
together, are coarser than the genus.

\subsection{The connection between factor equivalence and regulator constants}

Let $M$, $N$ be two $\Z[G]$-modules with the property that $M\otimes \Q\cong N\otimes \Q$,
let $i:M\rightarrow N$ be a map of $G$-modules with finite kernel and
cokernel. Fix a $\C$-valued bilinear pairing $\langle\cdot,\cdot\rangle$
on $N$ that is non-degenerate on $N/\tors$. The following immediate observation
is crucial for linking regulator constants with the notion of factorisability:
\begin{eqnarray*}
\det\left(\langle\cdot,\cdot\rangle\big|i(M)/\tors\right) & = &
[N/\tors: i(M)/\tors]^2\cdot
\det\left(\langle\cdot,\cdot\rangle\big|N/\tors\right)\\
& = &
\frac{[N:i(M)]^2}{|\ker i|^2}\cdot\frac{|M_{\tors}|^2}{|N_{\tors}|^2}\cdot
\det\left(\langle\cdot,\cdot\rangle\big|N/\tors\right).
\end{eqnarray*}

We deduce
\begin{lemma}\label{lem}
Let $M$, $N$ be two $\Z[G]$-modules such that $M\otimes \Q\cong N\otimes \Q$,
let $\Theta = \sum_H n_H H$ be a Brauer relation. Then
\[
\cC_\Theta(M) =
\prod_H\left(\frac{[N^H:i(M^H)]}{|\ker(i|_M^H)|}\cdot\frac{|M_{\tors}^H|}{|N_{\tors}^H|}\right)^{2n_H}\cdot\cC_\Theta(N)
\]
for any map $i:M\rightarrow N$ of $G$-modules with finite kernel and cokernel.
\end{lemma}
By combining this with Proposition \ref{prop:factorisability}, we obtain
\begin{corollary}\label{cor}
Two $\Z[G]$-modules $M$ and $N$ with the property that $M\otimes \Q\cong \N\otimes \Q$
are factor equivalent if and only if
\[
\cC_{\Theta}(M)/\cC_\Theta(N) = \prod_H\left(\frac{|M_{\tors}^H|}{|N_{\tors}^H|}\right)^{2n_H}
\]
for all Brauer relations $\Theta = \sum_H n_HH$. In particular, if $M$ and $N$ are
$\Z$-free and satisfy $M\otimes \Q\cong \N\otimes \Q$, then they are factor
equivalent if and only if $\cC_{\Theta}(M) = \cC_\Theta(N)$ for all $\Theta\in K(G)$.
\end{corollary}

\section{Galois module structure}\label{sec:results}
We shall now show by way of several examples how Lemma \ref{lem} and Corollary
\ref{cor} link known results on Galois module structures with each other.

Throughout this section, let $K/k$ be a finite Galois extension of number fields
with Galois group $G$. The ring of integers $\cO_K$, and its unit group $\cO_K^\times$
are both $\Z[G]$-modules. More generally, if $S$ is any $G$-stable set of places
of $K$ that contains the Archimedean places,
then the group of $S$-units $\cO_{K,S}^\times$ of $K$ is a $\Z[G]$-module.
It is a long standing and fascinating problem to determine the $G$-module
structure of these groups, e.g. by comparing it to other well-known $G$-modules
or by linking it to other arithmetic invariants.

A starting point is the observation that $\cO_K\otimes \Q\cong \Q[G]^{\oplus [k:\Q]}$
as $\Q[G]$-modules. Also, by Dirichlet's unit theorem,
$\cO_{K,S}^\times\otimes \Q\cong I_{K,S}\otimes \Q$, where
\[
I_{K,S} = \ker\left(\Z[S]\rightarrow \Z\right),
\]
with the map being the augmentation map that sends each $v\in S$ to 1.
It is therefore natural to compare the Galois module $\cO_K$ to $\Z[G]^{\oplus [k:\Q]}$
and $\cO_{K,S}^\times$ to $I_{K,S}$.

\subsection{Additive Galois module structure}
It had been known since E. Noether that $\cO_K$ lies in the same genus as
$\Z[G]^{\oplus [k:\Q]}$ if and only if $K/k$ is at most tamely ramified. The
following is therefore particularly interesting in the wildly ramified case:

\begin{theorem}[\cite{deSmit}, Theorem 3.2, see also \cite{Frohl}, Theorem 7 (Additive)]
We always have that $\cO_K$ is factor equivalent to $\Z[G]^{\oplus[k:\Q]}$.
\end{theorem}

We will now give a very short proof of this result in terms of regulator
constants. First, note that by Corollary \ref{cor} the statement is equivalent
to the claim that for any $\Theta\in K(G)$, $\cC_\Theta(\cO_K)
= \cC_\Theta(\Z[G]^{\oplus[k:\Q]})$. Since regulator constants are multiplicative in
direct sums of modules (\cite[Corollary 2.18]{tamroot}), and since
$\cC_\Theta(\Z[G])=1$ for all $\Theta\in K(G)$ (\cite[Example 2.19]{tamroot}),
we have reduced the proof of the theorem to showing that $\cC_\Theta(\cO_K)=1$
for all $\Theta\in K(G)$.

If we choose the pairing on $\cO_K$ defined by
\[
\langle a,b\rangle = \sum_\sigma \sigma(a)\sigma(b)
\]
with the sum running over all embeddings $\sigma: K\hookrightarrow \C$, then
the determinants on $\cO_K^H$, $H\leq G$, appearing in the definition of regulator
constants are nothing but the absolute discriminants $\Delta_{K^H}$. The fact that
these vanish in Brauer relations follows immediately from the
conductor-discriminant formula.

\subsection{Multiplicative Galois module structure}
As we have mentioned above, it is natural to compare $\cO_{K,S}^\times$ with
$I_{K,S}$, since they span isomorphic $\Q[G]$-modules. For $H\leq G$, let
$S(K^H)$ denote the set of places of $K^H$ below those in $S$, and let $h_S(K^H)$
denote the $S$-class number of $K^H$.

\begin{theorem}[\cite{deSmit}, Theorem 5.2, see also \cite{Frohl}, Theorem 7 (Multiplicative)]\label{thm:mult}
Fix an embedding $i: I_{K,S}\hookrightarrow \cO_{K,S}^\times$ of $G$-modules
with finite cokernel. For
$\fp\in S(K^H)$, let $f_\fp$ be its residue field degree in $K/K^H$, define
\[
n(H)=\prod_{\fp\in S(K^H)}f_\fp,\;\;\;l(H)=\lcm\{f_\fp \;|\;\fp\in S(K^H)\}. 
\]
Then the function
\[
H\mapsto[\cO_{K^H,S}^\times : i(I_{K,S})^H]\frac{n(H)}{h_S(K^H)l(H)} 
\]
is factorisable.
\end{theorem}

As in the additive case, we want to understand and to reprove this theorem
in terms of regulator constants. More specifically, we will show it to be
equivalent to

\begin{theorem}[\cite{classnorels}, Proposition 2.15 and equation (1)]\label{thm:propmult}
For $\fp\in S(k)$,
let $D_{\fp}$ be the decomposition group of a prime $\fP\in S$ above $\fp$
(well-defined up to conjugacy). For any Brauer relation
$\Theta = \sum_H n_H H\in K(G)$, we have
\[
\cC_\Theta(\cO_{K,S}^\times) = \frac{\cC_\Theta(\triv)}{\prod_{\fp\in S(k)}\cC_\Theta(\Z[G/D_\fp])}
\prod_H \left(\frac{w(K^H)}{h_S(K^H)}\right)^{2n_H},
\]
where $w(K^H)$ denotes the number of roots of unity in $K^H$, i.e. the size
of the torsion subgroup of $\cO_{K^H,S}^\times$.
\end{theorem}

Note that since $I_{K,S}$ is torsion free and $I_{K,S}\hookrightarrow \cO_{K,S}^\times$
is injective,
Proposition \ref{prop:factorisability} and Lemma \ref{lem} imply that Theorem
\ref{thm:mult} is equivalent to the
following statement: for any Brauer relation $\Theta=\sum_H n_H H$,
\[
\cC_\Theta(\cO_{K,S}^\times) = \cC_{\Theta}(I_{K,S})\prod_{H}
\left(\frac{w(K^H)n(H)}{h_S(K^H)l(H)}\right)^{2n_H}.
\]
The equivalence of Theorems \ref{thm:mult} and \ref{thm:propmult} will therefore
be established if we show that
\[
\cC_{\Theta}(I_{K,S}) = \frac{\cC_\Theta(\triv)}{\prod_{\fp\in S(k)}\cC_\Theta(\Z[G/D_\fp])}
\prod_H\left(\frac{l(H)}{n(H)}\right)^{2n_H}.
\]
This is just a linear algebra computation that we will not carry out in full detail,
since it is a combination of the computations of \cite{deSmit} and \cite{classnorels}.
Indeed, it is shown in \cite{deSmit} that under the embedding

\begin{eqnarray}\label{eq:embedding}
\Z[S(K^H)]\hookrightarrow \Z[S], \;\;\;\fp\mapsto \sum_{\fq\in S,\fq | \fp}f_\fp\fq
\end{eqnarray}

we have $[(I_{K,S})^H : I_{K^H,S}] = \frac{n(H)}{l(H)}$. So, instead of computing
\[
\cC_\Theta(I_{K,S}) = \prod_H \det\left(\frac{1}{|H|}\langle\cdot,\cdot\rangle \big| (I_{K,S})^H\right)^{n_H}
\]
for a suitable choice of pairing $\langle\cdot,\cdot\rangle$ on $I_{K,S}$, we may
compute
\begin{eqnarray}\label{eq:regconst}
\prod_H \det\left(\frac{1}{|H|}\langle\cdot,\cdot\rangle \big| I_{K^H,S}\right)^{n_H},
\end{eqnarray}
where $I_{K^H,S}$ is identified with a submodule of $I_{K,S}$ as in (\ref{eq:embedding}).
To do that, we note that for any $H\leq G$, $I_{K^H,S}$ is generated by
$\fp_1 - \fp_i$, $\fp_i\in S(K^H)\backslash\{\fp_1\}$
for any fixed $\fp_1\in S(K^H)$, and that there is a natural $G$-invariant
non-degenerate pairing on $I_{K,S}$
that makes the canonical basis of $\Z[S]$ orthonormal. It is now a straightforward
computation, which has essentially been carried out in \cite{classnorels},
to show that the quantity (\ref{eq:regconst}) is equal to
\[
\frac{\cC_\Theta(\triv)}{\prod_{\fp\in S(k)}\cC_\Theta(\Z[G/D_\fp])},
\]
as required.

\section{$K$-groups of rings of integers}\label{sec:K}
As another illustration of the connection we have established, we will give
an easy proof of an analogue of \cite[Theorem 5.2]{deSmit} for higher
$K$-groups of rings of integers. The main ingredient will be the compatibility
of Lichtenbaum's conjecture on leading coefficients of Dedekind zeta functions
at negative integers with Artin formalism, as proved in \cite{Bur-10}.

Let $n\geq 2$ be an integer. Let $S_1(F)$, respectively $S_2(F)$ denote the
set of real embeddings, respectively of representatives from each pair of
complex conjugate embeddings of a number field $F$,
and denote their cardinalities by $r_1(F)$, respectively $r_2(F)$. Denote
$S_2(F)\cup S_2(F)$ by $S_{\infty}(F)$.
It is shown in \cite{Borel} that the ranks of the higher $K$-groups or rings of
integers are as follows:
\[
\rk(K_{2n-1}(\cO_F)) = \leftchoice{r_1(F)+r_2(F)}{n\text{ odd}}{r_2(F)}{n\text{ even}.}
\]
Let $K/k$ be a finite Galois extension with Galois group
$G$, and let $S_r(K/k)$ denote the set of real places of $k$ that become complex
in $K$. For $\fp\in S_r(K/k)$, let $\epsilon_{\fp}$ denote the non-trivial
one-dimensional $\Q$-representation of the decomposition group $D_{\fp}$, which
has order 2.

By Artin's induction theorem, a rational representation of a finite group is
determined by the dimensions of the fixed subrepresentations under all subgroups of $G$. It
therefore follows that we have the following isomorphisms of Galois modules:
\begin{eqnarray}
K_{2n-1}(\cO_K)\otimes \Q & \cong &
\Q[S_\infty(K)]\nonumber\\
& \cong & \bigoplus_{\fp\in S_\infty(k)}\Q[G/D_\fp]\;\;\;\text{if $n$ is odd, and}\label{eq:KQodd}\\
K_{2n-1}(\cO_K)\otimes \Q & \cong &
\bigoplus_{\fp \in S_r(K/k)}\Ind_{G/D_\fp}\epsilon_{\fp}\oplus \bigoplus_{\fp\in S_2(k)}\Q[G]\nonumber\\
& \cong & \bigoplus_{\fp \in S_r(K/k)}\Q[G]\big/\Q[G/D_\fp] \oplus \bigoplus_{\fp\in S_2(k)}\Q[G]\;\;\;\text{if $n$ is  even.}\label{eq:KQeven}
\end{eqnarray}

We are thus led to compare, using the machine of factorisability,
the Galois module structure of $K_{2n-1}(\cO_K)$ with
$\Z[S_\infty(K)]$ when $n$ is odd, and with
\[
\bigoplus_{\fp \in S_r(K/k)}\Ind_{G/D_\fp}\left(\epsilon_\fp\right) \oplus \bigoplus_{\fp\in S_2(k)}\Z[G]
\]
when $n$ is even. Here and elsewhere, we write $\epsilon_\fp$ interchangeably
for the rational representation and for the unique (up to isomorphism) $\Z$-free
$\Z[D_\fp]$-module inside it.

\begin{theorem}
Let $K/k$ be a finite Galois extension of number fields with Galois group $G$,
let $n\geq 2$ be an integer. Then the function
\[
H\mapsto \frac{[K_{2n-1}(\cO_{K})^H:i(M)^H]}{|K_{2n-2}(\cO_{K^H})|}
\]
is factorisable at all odd primes, where
\begin{eqnarray*}
M & = & \Z[S_\infty(K)]\cong \bigoplus_{\fp\in S_\infty(k)}\Z[G/D_\fp]\;\;\;\;\text{if $n$ is odd, and}\\
M & = & \bigoplus_{\fp \in S_r(K/k)}\Ind_{G/D_\fp}\left(\epsilon_\fp\right) \oplus \bigoplus_{\fp\in S_2(k)}\Z[G] \;\;\;\text{if $n$ is even},
\end{eqnarray*}
and where $i: M \hookrightarrow K_{2n-1}(\cO_K)$
is any inclusion of $G$-modules.
\end{theorem}

\begin{proof}
  Proposition \ref{prop:factorisability} and Lemma \ref{lem} imply that the
  assertion of the theorem is equivalent to the claim that
  for any Brauer relation $\Theta=\sum_H n_H H$,
  \begin{eqnarray*}
    1 & =_{2'} & \prod_H \frac{[K_{2n-1}(\cO_{K})^H:i(M)^H]^{2n_H}}{|K_{2n-2}(\cO_{K^H})|^{2n_H}}\\
    & =_{2'} & \frac{\cC_{\Theta}(M)}{\cC_{\Theta}(K_{2n-1}(\cO))}\cdot
    \prod_H \left(\frac{|K_{2n-1}(\cO_{K})^H_{\tors}|}{|K_{2n-2}(\cO_{K^H})|}\right)^{2n_H},
  \end{eqnarray*}
  where $=_{2'}$ means that the two sides have the same $p$-adic valuation for
  all odd primes $p$.
  
  Now, for any odd prime $p$ and any subgroup $H\leq G$, we have
  \[
    (K_{2n-1}(\cO_K)\otimes \Z_p)^H \cong K_{2n-1}(\cO_{K^H})\otimes \Z_p.
  \]
  This is a consequence of the Quillen--Lichtenbaum conjecture
  (see e.g. \cite[Proposition 2.9 and the discussion preceding it]{Kol-02}),
  which is known
  to follow from the Bloch--Kato conjecture, which in turn is now a theorem of
  Rost, Voevodsky, and Weibel \cite{Rost, Voevodsky, Weibel}.
  Moreover, it follows from \cite{Bur-10} (see \cite[equation (2.6)]{SB}) that
  \[
  \prod_H\left(\frac{|K_{2n-1}(\cO_{K^H})_{\tors}|}{|K_{2n-2}(\cO_{K^H})|}\right)^{2n_H} =_{2'} \cC_\Theta(K_{2n-1}(\cO_K)).
  \]
  Putting this together, we see that the assertion of the theorem
  is equivalent to the claim that $\cC_\Theta(M)=_{2'}1$ for all Brauer relations
  $\Theta$.
  But $\cC_\Theta(M)=1$ (not just up to powers of 2) by
  \cite[Corollary 2.18 and Proposition 2.45 (2)]{tamroot}, and because
  cyclic groups have no non-trivial Brauer relations.
\end{proof}

\end{document}